\documentclass[11pt]{amsart}
\usepackage{geometry}                
\usepackage{amsmath, amssymb, amsthm,amsfonts}
\usepackage{xcolor}
\usepackage[applemac]{inputenc}
\definecolor{myurlcolor}{rgb}{0,0,0.4}
\definecolor{mycitecolor}{rgb}{0,0.5,0}
\definecolor{myrefcolor}{rgb}{0.5,0,0}
\usepackage[pagebackref,colorlinks,linkcolor=myrefcolor,citecolor=mycitecolor,urlcolor=myurlcolor,hypertexnames=true]{hyperref}
\usepackage{amsrefs}
\usepackage{graphicx}
\usepackage{pgfplots}
\usepackage{enumitem}
\usepackage{microtype}
\usepackage{cleveref}
\usepgfplotslibrary{fillbetween}
\theoremstyle{plain}
\newtheorem{theorem}{Theorem}[section]

\newtheorem{corollary}[theorem]{Corollary}
\newtheorem{proposition}[theorem]{Proposition}
\theoremstyle{remark}

\newtheorem{remark}[theorem]{Remark}
\theoremstyle{definition}
\newtheorem{definition}[theorem]{Definition}
\newtheorem{example}[theorem]{Example}

\title{Classifying Linear Matrix Inequalities via Abstract Operator Systems}
\author{Martin Berger}
\author{Tom Drescher}
\author{Tim Netzer}
\address{Universit\"at Innsbruck \\ Faculty of Mathematics, Computer Science and Physics\\ Department of Mathematics}
\date{\today}                                          
\begin{document}
\maketitle
\begin{abstract}
We systematically study how properties of abstract operator systems help classifying linear matrix inequality definitions of sets. Our main focus is on polyhedral cones,  the $3$-dimensional Lorentz cone, where we can completely describe all defining linear matrix inequalities, and on the cone of positive semidefinite matrices. Here we use results on isometries between matrix algebras to describe linear matrix inequality definitions of relatively small size. We conversely use the theory of operator systems to characterize special such isometries.
\end{abstract}
\section{Introduction}
A  spectrahedron is, by definition, the solution set of a linear matrix inequality. We will restrict ourselves to convex cones throughout this paper, so spectrahedra have the form $$\{a\in\mathbb R^d\mid a_1M_1+\cdots +a_dM_d\geqslant 0\},$$ where $M_1,\ldots, M_d\in{\rm Her}_n(\mathbb C)$ are Hermitian matrices, and $\geqslant 0$ means that the matrix is positive semidefinite. All polyhedra are spectrahedra, but there are many more.  Spectrahedra are precisely the feasible sets of semidefinite programming, and thus deciding whether a set is a spectrahedron is a relevant problem for optimization. It is also a surprisingly hard problem, with deep connections to the theory of determinantal representations of polynomials and convex algebraic geometry. See \cite{blek, nepl}  for many more details on this exciting topic.

Here we ask a slightly different question: Given a spectrahedron, by how many different linear matrix inequalities can it be defined? It has first been observed in \cite{hkm2} that this question relates strongly  to the so-called free (=noncommutative) extension of the spectahedron. These free extensions are, on the other hand, almost the same as abstract operator systems, a well-studied concept in operator algebra (see \cite{pau}, for example).

In this paper we  systematically study how different properties of abstract operator systems lead to results on linear matrix inequality representations of sets. We prove several new results on abstract operator systems, and also apply existing ones to gain new insights into spectrahedral representations. Our main focus is on polyhedral cones, the circular cone,  and on cones of positive semidefinite matrices.

This paper is structured as follows. In \Cref{sec:prel} we explain all concepts  that we will use,  \Cref{sec:main}  contains our main results. In \Cref{ssec:class} we first explain how properties of operator systems shed light on linear matrix inequality definitions of sets in general.  \Cref{ssec:gen} contains results on the largest operator system over a general convex cone. These results are then strengthened for special classes of cones. In particular, \Cref{ssec:poly}  deals with polyhedral cones, and \Cref{ssec:disk} is devoted to the Lorentz cone in $\mathbb R^3.$ Here we can completely  classify all defining linear matrix inequalities. By proving a result on the largest operator system over the Lorentz cone, we also show that there exist an infinite sequence of weaker and weaker such definitions. Finally, \Cref{ssec:psd} deals with the cone of positive semidefinite (psd) matrices. Since a linear matrix inequality definition of this cone is the same as a unital $*$-isometry between matrix algebras, we can use existing results to obtain insights into operator systems and linear matrix inequality definitions of the psd cone. We then use the operator systems approach to characterize special $*$-isometries between matrix algebras.

\subsection*{Acknowledgement}We thank Mario Kummer for suggesting a very elegant alternative proof of our \Cref{prop:deriv} below, and for allowing us to include it here.
 Martin Berger greatfully acknowledges financial suppport by the Austrian Academy of Sciences (\"OAW) through a DOC Scholarship.
\section{Preliminaries}\label{sec:prel}
In this section we collect all important notions and facts that we will use for our main results. 

\subsection{Abstract Operator Systems}
Throughout $\mathcal V$ denotes a finite-dimensional $\mathbb C$-vector space with an involution $*$.  For each $s\geqslant 1$, the space $${\rm Mat}_{s}(\mathbb C)\otimes \mathcal V={\rm Mat}_s(\mathcal V)$$ inherits a canonical involution (entrywise $*$ and transposition), and the $\mathbb R$-vector space   of Hermitian elements (i.e.\ fixed points of the involution) is denoted by  ${\rm Her}_s(\mathcal V)$. All spaces are equipped with the Euclidean norm/topology. A convex cone is called {\it proper}, if it is closed, has nonempty interior, and does not contain a full line.

\begin{definition}
An {\it abstract operator system} $\mathcal C=\left(\mathcal C_s\right)_{s\geqslant 1}$  on $\mathcal{V}$ consists of a proper convex cone $\mathcal{C}_s\subseteq {\rm Her}_s(\mathcal{V})$ for each $s\geqslant 1$,  such that
$$A\in \mathcal{C}_s, V\in {\rm Mat}_{s,t}(\mathbb C) \ \Rightarrow\ V^*AV\in \mathcal{C}_t.$$ Often, a fixed interior point $u\in\mathcal C_1$ is also considered part of the structure.  We call $\mathcal C_s$ the {\it cone at level $s$} of the abstract operator system $\mathcal C$.
\end{definition}

By the Choi--Effros Theorem (\cite{choi}, see  also \cite{pau}), for every abstract operator system $\mathcal C$ there exists a Hilbert space $\mathcal H$ and a $*$-linear map $\varphi\colon \mathcal V\rightarrow \mathbb B(\mathcal H)$ with  $\varphi(u)={\rm id}_{\mathcal H}$, such that for all $s\geqslant 1$ and $A\in {\rm Her}_s(\mathcal V)$,
$$
A\in \mathcal{C}_s \:\Leftrightarrow\: ({\rm id}\otimes\varphi)(A)\geqslant 0,
$$
where $\geqslant 0$ denotes positive semidefiniteness (psd).
On the right-hand side, we use the canonical identification
$$
 {\rm Mat}_s(\mathbb C)\otimes \mathbb B(\mathcal H)={\rm Mat}_s(\mathbb B(\mathcal H))= \mathbb B(\mathcal H^s)
$$
to define positivity of the operator. Such a mapping $\varphi$ is called a {\it concrete realization} or just {\it realization} of the operator system $\mathcal C$.

\subsection{Free Spectrahedra and Numerical Ranges} Special classes of abstract operator systems are of particular interest to us.
\begin{definition}
($i$) An abstract operator system is called  a {\it free spectrahedron,} if it has a realization with $\dim{\mathcal{H}}<\infty$. 

($ii$) We call an operator system  a {\it free numerical range,} if it is finitely generated.  That means it contains finitely many elements, for which it is the smallest one containing them.
\end{definition}

\begin{remark}
Since we assume  $\mathcal V$ to be  finite-dimensional,  we may  assume $\mathcal V=\mathbb C^d$ with the canonical involution, and thus  $${\rm Mat}_s(\mathcal V)={\rm Mat}_s(\mathbb C)^d, \quad {\rm Her}_s(\mathcal V)={\rm Her}_s(\mathbb C)^d.$$ Then a realization of an abstract operator system just consists of self-adjoint operators $M_1,\ldots, M_d\in\mathbb B(\mathcal H)$ with $u_1M_1+\cdots +u_dM_d={\rm id}_{\mathcal H}$, such that $$
(A_1,\ldots, A_d)\in \mathcal{C}_s \:\Leftrightarrow\: A_1\otimes M_1+\cdots + A_d\otimes M_d\geqslant 0.
$$
Finite-dimensional realizability  means that the $M_i$ can be taken as Hermitian matrices. In this case we denote the operator system/free spectrahedron by  $$\mathcal S(M_1,\ldots,M_d)=\left(\mathcal S_s(M_1,\ldots, M_d)\right)_{s\geqslant 1}.$$ In particular, we have $$\mathcal S_1(M_1,\ldots, M_d)=\{a\in\mathbb R^d\mid a_1M_1+\cdots +a_dM_d\geqslant 0\},$$ which is the definition of a classical spectrahedron, the feasible set of a semidefinite program.
\end{remark}

\begin{remark} Abstract operator systems are closed under block-diagonal sums: $$A\in\mathcal C_s, B\in\mathcal C_t \ \Rightarrow \ A\oplus B\in \mathcal C_{s+t}.$$ So every free numerical range is generated by a single element $M\in\mathcal C_r.$ This implies that  $$\mathcal C_s=\left\{ \sum_{j=1}^{n} V_j^*MV_j\mid n\geqslant 1, V_j\in {\rm Mat}_{r,s}(\mathbb C)\right\}$$ holds for all $s\geqslant 1.$ In particular, if $M=(M_1,\ldots, M_d)\in {\rm Her}_r(\mathbb C)^d$ and $u_1M_1+\cdots +u_dM_d=I_r$ holds for some $u\in\mathbb R^d$, then $$\{a\in\mathcal C_1\mid u^ta=1\}=\left\{ \sum_j(v_j^*M_1v_j,\ldots, v_j^*M_dv_j)\mid v_j\in\mathbb C^r, \sum_j v_j^*v_j=1\right\}$$ is the convex hull of the so-called {\it joint numerical range} of $M_1,\ldots, M_d$. 
If a free numerical range is generated by $(M_1,\ldots, M_d)\in {\rm Her}_r(\mathbb C)^d,$ we denote it by $$\mathcal W(M_1,\ldots,M_d)=\left(\mathcal W_s(M_1,\ldots, M_d)\right)_{s\geqslant 1}.$$
\end{remark}

The following result is the separation method from \cite{effros}, formulated in our context:
\begin{theorem}\label{thm:sep} Every abstract operator system on $\mathcal V$ is an intersection  of free specatrahedra.
\end{theorem}

\subsection{Smallest and Largest Operator Systems}
Let $C\subseteq\mathbb R^d$ be a proper convex  cone. We are interested in operator systems $\mathcal{C}=\left(\mathcal C_s\right)_{s\geqslant 1}$ with $\mathcal C_1=C$. We call this an {\it operator system over} $C$. There is a largest such operator system $\mathcal C(C)^{\rm lrg}$, defined by 
$$
\mathcal C_s(C)^{\rm lrg}:=\left\{ (A_1,\ldots, A_d)\in{\rm Her}_s(\mathbb C)^d \mid \forall v\in \mathbb C^s\colon (v^*A_1v,\ldots, v^*A_dv)\in C\right\}.
$$
There is also a smallest such operator system $\mathcal C(C)^{\rm sml}$, defined by 
$$
\mathcal C_s(C)^{\rm sml}:=\left\{\sum_{j=1}^n  P_j \otimes c_j \mid n\geqslant 1,c_j\in C, P_j\in{\rm Her}_s(\mathbb C), P_j\geqslant 0\right\}.
$$ 
These systems are largest/smallest in the sense that for any operator systems $\left(  \mathcal C_s\right)_{s\geqslant 1}, \left(  \mathcal D_s\right)_{s\geqslant 1}$ with $ \mathcal C_1\subseteq C\subseteq \mathcal D_1$, we have $ \mathcal C_s(C)^{\rm sml}\subseteq \mathcal D_s$ and $\mathcal C_s\subseteq  \mathcal C_s(C)^{\rm lrg}$ for all $s\geqslant 1.$ For details and proofs of these statements see \cite{pauto,fritz}. 

\subsection{Duality} Given an abstract operator system $\mathcal C=(\mathcal C_s)_{s\geqslant 1}$ on $\mathbb C^d,$ its {\it free dual} $\mathcal C^{\vee_{\rm fr}}$ is defined by 
$$\mathcal C^{\vee_{\rm fr}}_s:=\left\{ (A_1,\ldots, A_d)\in {\rm Her}_s(\mathbb C)^d\mid \forall s\geqslant 1, B\in\mathcal C_s\colon  \sum_{i=1}^d B_i\otimes A_i\geqslant 0 \right\}.$$ The following result  summarizes the most important properties of the free dual. Proofs can for example be found in  \cite{bene, hkm}.
\begin{theorem}\label{thm:dual}
Let $C\subseteq \mathbb R^d$ be a proper convex cone, $\mathcal C=(\mathcal C_s)_{s\geqslant 1}$ an abstract operator system and $\mathcal S(M_1,\ldots, M_d)$ a free spectrahedron on $\mathbb C^d$.
  \begin{enumerate}
 \item[$(i)$] $\mathcal{C}^{\vee_{\rm fr}}$ is an abstract operator system.
 \item[$(ii)$] $(\mathcal C^{\vee_{\rm fr}})^{\vee_{\rm fr}}=\mathcal C$.
\item [$(iii)$] $(\mathcal C(C)^{\rm sml})^{\vee_{\rm fr}}=\mathcal C(C^\vee)^{\rm lrg}$ and $(\mathcal C(C)^{\rm lrg})^{\vee_{\rm fr}}=\mathcal C(C^\vee)^{\rm sml},$ where $C^\vee$ denotes the dual cone of $C\subseteq \mathbb R^d$.
\item[$(iv)$] $\mathcal S(M_1,\ldots, M_d)^{\vee_{\rm fr}}=\mathcal W(M_1,\ldots, M_d)$ and $\mathcal W(M_1,\ldots, M_d)^{\vee_{\rm fr}}=\mathcal S(M_1,\ldots, M_d).$
 In particular, $\mathcal S(M_1,\ldots, M_d)^{\vee_{\rm fr}}\subseteq \mathcal S(M_1,\ldots, M_d)$ if and only if $\sum_{i=1}^d M_i\otimes M_i\geqslant 0$.
 \end{enumerate}
\end{theorem}

\subsection{Containment}
For two abstract operator systems $\mathcal C=\left(\mathcal C_s\right)_{s\geqslant 1}$ and $\mathcal D=\left(\mathcal D_s\right)_{s\geqslant 1}$ on $\mathcal V,$ we write $\mathcal C\subseteq \mathcal D$ if $\mathcal C_s\subseteq\mathcal D_s$ holds for every $s\geqslant 1.$ Note that $\mathcal C_1\subseteq\mathcal D_1$ does in general not imply $\mathcal C\subseteq\mathcal D$. In the following result from \cite{hkm2}, equivalence of ($i$) and ($ii$) follows from \Cref{thm:dual}, equivalence of ($ii$) and ($iii$) is clear. 

\begin{theorem}\label{thm:cont}
Let $(M_1,\ldots, M_d)\in {\rm Her}_{r_1}(\mathbb C)^d$ and $(N_1,\ldots, N_d)\in{\rm Her}_{r_2}(\mathbb C)^d$ be such that $\mathcal S(M_1,\ldots, M_d)$ and $\mathcal S(N_1,\ldots, N_d)$ are free spectrahedra.Then the following are equivalent:
\begin{enumerate}
\item[$(i)$]  $\mathcal S(M_1,\ldots, M_d)\subseteq \mathcal S(N_1,\ldots, N_d)$
\item[$(ii)$]  $\mathcal W(N_1,\ldots, N_d)\subseteq \mathcal W(M_1,\ldots, M_d)$ 
\item[$(iii)$]  There exist $V_j\in {\rm Mat}_{r_1,r_2}(\mathbb C)$ with  $$\sum_j V_j^*M_iV_j=N_i$$ for all $i=1,\ldots, d.$
\end{enumerate}  
\end{theorem}

\subsection{Free Spectrahedral Shadows}
If $\pi\colon\mathbb R^{d+e}\to\mathbb R^d$ is the canonical projection, and $C\subseteq \mathbb R^{d+e}$ a (classical) spectrahedral cone, then $\pi(C)$ is called a {\it spectrahedral shadow.} We can extend this notion to free spectrahedra as follows.
For each $s\geqslant 1$, consider the projection 
\begin{align*}\pi_s\colon {\rm Her}_{s}(\mathbb C)^{d+e}&\to   {\rm Her}_{s}(\mathbb C)^{d} \\ 
(A_1,\ldots, A_d,B_1,\ldots, B_e)&\mapsto (A_1,\ldots, A_d).
\end{align*} If $\mathcal C=\mathcal S(M_1,\ldots, M_d,N_1,\ldots, N_e)$ is a free spectrahedron on $\mathbb C^{d+e}$, then $$\pi_{\rm fr}(\mathcal S(M_1,\ldots, M_d,N_1,\ldots, N_e):=\left(\pi_s\left(\mathcal S_s(M_1,\ldots, M_d,N_1,\ldots, N_e)\right)\right)_{s\geqslant 1}$$
is is  called a {\it free spectrahedral shadow} (or {\it spectrahedrop} in  \cite{hkm}). 
The following result was proven in \cite{hkm}, extending the same result for classical spectrahedral shadows to the free setup.
\begin{theorem}\label{thm:dualshadow} The free dual of a free spectrahedral shadow is a free spectrahedral shadow. In particular, 
each free numerical range is a free spectrahedral shadow.
\end{theorem}
\section{Results}\label{sec:main}

We are now prepared for the main results of our paper. In \Cref{ssec:class} we first explain a general relation between the smallest/largest operator system, and linear matrix inequality definitions of a cone. We take this as a motivation to study the largest operator system over a quite general cone in \Cref{ssec:gen}, showing that its interior is always covered by free spectrahedral shadows over the cone. In \Cref{ssec:poly} we briefly describe and review what we know about linear matrix inequality definitions of polyhedra. \Cref{ssec:disk} is devoted to the circular cone in $\mathbb R^3$, which has some surprising properties. They allow us to completely classify all of its linear matrix inequality definitions. We also show that it admits a whole sequence of weaker and weaker linear matrix inequality definitions. Finally, \Cref{ssec:psd} deals with the cone of positive semidefinite matrices. We show that small linear matrix inequality definitions always give rise to one of the two standard ones (identity and transposition), and then characterize when this happens also for larger sizes.

\subsection{Operator Systems and Classification of Linear Matrix Inequalities}\label{ssec:class}
We start with a general explanation   how the theory of abstract operator systems, and in particular the smallest and largest operator system, relate to the problem of classifying matrix tuples $(M_1,\ldots, M_d)$ that realize a  given proper convex cone $C\subseteq\mathbb R^d$ as a spectrahedron: $$C=\mathcal S_1(M_1,\ldots, M_d).$$ We also say that $M_1,\ldots, M_d$ {\it define $C$ by a linear matrix inequality} in this case.

First assume $$\mathcal C(C)^{\rm lrg}=\mathcal S(N_1,\ldots, N_d)$$ is a free spectrahedron.  This means that $(N_1,\ldots, N_d)$ provide  the {\it weakest} linear matrix inequality definition of $C$, i.e.\ the $N_i$  arise from every other linear matrix inequality for $C$ in the sense of  \Cref{thm:cont} ($iii$).  Similarly, if $\mathcal C(C)^{\rm sml}=\mathcal S(M_1,\ldots, M_d)$ is a free spectrahedron, then $(M_1,\ldots, M_d)$ provide the  {\it strongest} definition, from which every other can be obtained.  

The main result from \cite{fritz} partially classifies in which cases this holds.
Indeed, $\mathcal C(C)^{\rm lrg}$ is a free spectrahedron if and only if $C$ is polyhedral. In this case, $\mathcal C(C)^{\rm sml}$ is a free spectrahedron if and only if $C$ is a simplex cone. Note that the weakest definition of a polyhedral cone is by a diagonal linear matrix inequality, containing the linear inequalities for the facets of $C$ on the diagonal. In the case of a simplex cone, this is the strongest definition at the same time.

So now assume that  $\mathcal C(C)^{\rm sml}$ is not a free spectrahedron. So whenever  $C=\mathcal S_1(N_1,\ldots, N_d)$ we have  $$\mathcal C(C)^{\rm sml}\subsetneq \mathcal S(N_1,\ldots, N_d).$$  By  \Cref{thm:sep} we can find matrices $(M_1,\ldots, M_d)$  such that $$\mathcal C(C)^{\rm sml}\subsetneq \mathcal S(M_1,\ldots, M_d)\subsetneq \mathcal S(N_1,\ldots, N_d).$$  So the new matrices are not sums of compressions of the old ones,  only vice versa. We can iterate the process and obtain a sequence of stronger and stronger  linear matrix inequality definitions of $C$, of which none is a sums of compressions of the previous one. 

On the other hand, if $\mathcal C(C)^{\rm lrg}$ is not a free spectrahedron (i.e.\ if $C$ is not polyhedral), the situation might be more complicated. Although we don't have an example, it might happen that every free spectrahedron over $C$ is contained in one of finitely many maximal ones.  In that case, $C$ would possess finitely many weak linear matrix inequality definitions, in the sense that every other gives rise to one of those by sums of compressions.

However, if this is not the case,  we obtain  a sequence of weaker and weaker linear matrix inequality definitions for $C$, such that none gives rise to the previous one by sums of compressions. This is for example the case for the circular cone, as we will  show in \Cref{ssec:disk}.
 
\subsection{General Cones}\label{ssec:gen} We first prove a general result  about the largest operator system over a general spectrahedron/spectrahedral shadow. All notions such as union, interior...\ are understood levelwise.

\begin{theorem}\label{thm:gencon}Let  $C\subseteq\mathbb R^d$ be a proper convex cone.

$(i)$ If $C^\vee$ is a  spectrahedron, then the union over all free numerical ranges  over $C$  covers the interior of $\mathcal C(C)^{\rm lrg}$.

$(ii)$ If $C$ is a spectrahedral shadow, then the union over all free spectrahedral shadows  over $C$ covers the interior of $\mathcal C(C)^{\rm lrg}$.
\end{theorem}
\begin{proof}
($i$) From \Cref{thm:sep} we know that \begin{equation}\label{eq:int}\mathcal C(C^{\vee})^{\rm sml}=\bigcap_{i\in I} \mathcal C^{(i)}\end{equation} is the intersection over all free spectrahedra $\mathcal C^{(i)}$ over  $C^{\vee}$.   By duality  we  obtain $$\mathcal C(C)^{\rm lrg}=\left(\mathcal C(C^{\vee})^{\rm sml}\right)^{\vee_{\rm fr}}=\left(\bigcap_{i\in I} \mathcal C^{(i)}\right)^{\vee_{\rm fr}}=\left(\bigcup_{i\in I} {\mathcal C^{(i)}}^{\vee_{\rm fr}}\right)^{\vee_{\rm fr}\vee_{\rm fr}}.$$ By \Cref{thm:dual}, each ${\mathcal C^{(i)}}^{\vee_{\rm fr}}$ is a free numerical range. Since the free numerical ranges over $C$ form a poset, the free double dual of the union  is just its (levelwise) closure, which proves the claim. 

($ii$) is proven similarly, but letting the intersection in (\ref{eq:int}) run over all free spectrahedral shadows over $C$.
 \end{proof}

The next result is a characterization  of all spectrahedra that contain a given cone in terms of a largest operator system.

\begin{proposition}
    \label{thm:genspeccontain}
    Let $C\subseteq \mathbb{R}^d$ be a proper convex cone and let $N_1,\dots,N_d\in\mathrm{Her}_s(\mathbb{C})$ be Hermitian matrices. Then $C\subseteq\mathcal{S}_1(N_1,\dots,N_d)$ if and only if $(N_1,\dots,N_d)\in\mathcal{C}(C^{\vee})^{\mathrm{lrg}}$.
\end{proposition}
\begin{proof}
    We have the following equivalences:
    \begin{align*}
        C\subseteq\mathcal{S}_1(N_1,\dots,N_d) &\Leftrightarrow \forall a\in C\colon a_1N_1+\dots+a_dN_d\geqslant 0\\
        &\Leftrightarrow \forall v\in\mathbb{C}^s,a\in C\colon a_1(v^*N_1v)+\dots+a_d(v^*N_dv)\geqslant 0\\
        &\Leftrightarrow \forall v\in\mathbb{C}^s\colon (v^*N_1v,\dots,v^*N_dv)\in C^{\vee}\\
        &\Leftrightarrow (N_1,\dots,N_d)\in\mathcal{C}(C^{\vee})^{\mathrm{lrg}}.\qedhere
    \end{align*}
\end{proof}

\subsection{Polyhedral and Simplicial Cones}\label{ssec:poly} The following result on linear matrix inequality definitions of polyhedra is a  direct consequences of the above described facts. \begin{corollary} Let $P\subseteq\mathbb R^d$ be polyhedral, but not a simplex cone.
Then there exists a sequence of linear matrix inequality definitions of $P$, of which none is a sums of compressions of the previous one.
\end{corollary}

The following result on linear matrix inequality definitions for polyhedal cones is \cite{sanro}, and it also follows from the fact  that the largest operator system over a polyhedral cone is a free spectrahedron, see \cite{nepl}. 
It states that if the spectrahedron $C=\mathcal S_1(M_1,\ldots, M_d)\subseteq \mathbb R^d$ has a face of dimension $d-1$, then there exists $A\in{\rm GL}_r(\mathbb C)$ such that $$A^*M_iA=\left(\begin{array}{cc}m_i & 0 \\0 & \tilde M_i\end{array}\right)$$ for all $i=1,\ldots, d,$ and the supporting hyperplane at the face is defined by the equation $$m_1x_1+\cdots +m_dx_d= 0.$$ So for polyhedral $P$, any defining linear matrix inequality splits off the canonical diagonal/weakest one. The other block then defines some larger spectrahedron, about which we cannot say something in general. But in case the polyhedral cone is a simplex, we can. The main reason why simplex cones are so easy, is that  the positive orthant $\mathbb R_{\geqslant 0}^d$ is the only simplex cone in $\mathbb R^d,$ up to isomorphism. We can thus restrict to that case, without loss of generality.

\begin{proposition}Let $N_1,\ldots, N_d\in{\rm Her}_n(\mathbb C)$ with $N_1+\cdots +N_d=I_n.$
\begin{itemize}
\item[$(i)$] $\mathbb R_{\geqslant 0}^d\subseteq\mathcal S_1(N_1,\ldots, N_d)$ if and only if all $N_i$ are positive semidefinite.
\item[$(ii)$] $\mathbb R_{\geqslant 0}^d=\mathcal S_1(N_1,\ldots, N_d)$ if and only if there is a unitary $U\in{\rm Mat}_n(\mathbb C)$ with  $$U^*N_iU=\left(\begin{array}{cc}E_{ii} & 0 \\0 & \tilde N_i\end{array}\right)$$ and $\tilde N_i\geqslant 0$ for all $i=1,\ldots, d.$ Here, $E_{ii}$ denotes the standard matrix unit of size $d$.
\end{itemize}
\end{proposition}
\begin{proof} ($i$) is obvious. For ($ii$) note that 
 $\mathbb R_{\geqslant 0}^d=\mathcal S_1(N_1,\ldots, N_d)$ just means, in addition to all $N_i\geqslant 0$, that $\sum_{i\neq j}N_i$ is rank deficient, for each $j=1,\ldots, d.$ For matrices with the stated property, this is clearly fulfilled. Conversely, we diagonalize  $$U^*(N_2+\cdots +N_{d})U=\left(\begin{array}{cc}0 & 0 \\0 & *\end{array}\right)=U^*N_2U+\cdots +U^*N_{d}U$$ with a suitable unitary $U$, and obtain $$U^*N_iU= \left(\begin{array}{cc}0 & 0 \\0 & *\end{array}\right)\ \mbox{ for } i=2,\ldots, d\ \mbox{ and }\  U^*N_1U=\left(\begin{array}{cc}1 & 0 \\0 & *\end{array}\right).$$ Repeating this argument with the lower right blocks proves the claim.
\end{proof}
\subsection{The Circular Cone}\label{ssec:disk}
Throughout this section let $$D:=\left\{ (a_0,a_1,a_2)\in \mathbb R^3\mid 0\leqslant a_0, a_1^2+a_2^2\leqslant a_0^2\right\}$$ be the standard circular cone in $\mathbb R^3$.
From \cite{choi2} it follows that  $\mathcal C(D)^{\rm sml}$ is a free spectrahedron  (see also  \cite[Example 4.10]{fritz} or \cite[Corollary 14.15]{HKMS}), defined by the linear matrix inequality $$A_0\otimes I_2+ A_1\otimes \underbrace{\left(\begin{array}{cc}1 & 0 \\0 & -1\end{array}\right)}_{=:M_1} +A_2\otimes \underbrace{\left(\begin{array}{cc}0 & 1 \\1 & 0\end{array}\right)}_{=:M_2}\geqslant 0.$$
We can use this fact to characterize all linear matrix inequalities that define $D$. This is precisely the dual statement of \cite[Theorem 4.5]{woe}. Since our proof is different, in particular the one  of ($i$), we include it for completeness.

\begin{theorem}\label{thm:circcone}Let  $N_0,N_1,N_2\in{\rm Her}_n(\mathbb C)$ be arbitrary.
\begin{itemize}
\item[$(i)$]  We have $D\subseteq \mathcal S_1(N_0,N_1,N_2)$ if and only if there are $P,Q\in{\rm Mat}_{r,n}(\mathbb C)$ with $$N_0=P^*P+Q^*Q\quad\mbox{and}\quad N_1=P^*Q+Q^*P\quad \mbox{and}\quad  N_2=i(Q^*P-P^*Q).$$

\item[$(ii)$]  We have $D=\mathcal S_1(N_0,N_1,N_2)$ if and only if we can choose $r=n-1$ in $(i)$, and this happens already if the boundary of $S_1(N_0,N_1,N_2)$ contains at least $n+1$ rays from $D$.
\end{itemize}
\end{theorem}

\begin{proof}
($i$)  From $D\subseteq \mathcal S_1(N_0,N_1,N_2)$ we get  $$\mathcal S(I_2, M_1,M_2) = \mathcal C(D)^{\rm sml}\subseteq \mathcal S(N_0,N_1,N_2),$$ and from \Cref{thm:cont} we obtain  $V_1,\ldots, V_r\in{\rm Mat}_{2,n}(\mathbb C)$ with $$N_0=\sum_{j=1}^r V_j^*V_j,\quad N_1=\sum_{j=1}^r V_j^*M_1V_j, \quad  N_2=\sum_{j=1}^r V_j^*M_2V_j.$$  We set $$P= \frac{1}{\sqrt{2}}\left(\begin{array}{c}(1,-i)\cdot V_1 \\\vdots \\(1,-i)\cdot V_r\end{array}\right), \ Q=\frac{1}{\sqrt{2}}\left(\begin{array}{c}(1,i)\cdot V_1 \\\vdots \\(1,i)\cdot V_r\end{array}\right)\in {\rm Mat}_{r,n}(\mathbb C)$$ and obtain the desired representation.  Conversely, every such representation shows $D\subseteq \mathcal S_1(N_0,N_1,N_2).$ Indeed for $a,b\in\mathbb R$ with $a^2+b^2=1$ we set $z:=a+bi\in\mathbb C$ and compute \begin{equation}\label{eq:1}0\leqslant (zP+Q)^*(zP+Q)=N_0+\overline zP^*Q+zQ^*P=N_0+aN_1+bN_2.\end{equation}

($ii$) Let $\bar  P,\bar  Q$ provide a decomposition for $N_0,N_1,N_2$ as in ($i$). We further assume that $\bar Q^*\bar Q$ is maximal w.r.t. $\leqslant$ among all such decompositions (for fixed $r$).
Now assume  we have $(1,a,b)\in\partial \mathcal S_1(N_0,N_1,N_2)$  for some $a,b\in\mathbb R$ with $a^2+b^2=1$. Then  (\ref{eq:1})  implies that for $z=a+bi,$ the matrix $$z\bar P+\bar Q\in{\rm Mat}_{r,n}(\mathbb C)$$ has rank at most $n-1.$ If this happens for at least  $n+1$ different numbers $z$,  $\bar Q$ has rank at most $n-1.$
We now compute  the singular value decomposition $\bar Q=U\Sigma V^*,$ and observe that $$\tilde P:=U^*\bar P \mbox{ and } \ \tilde Q:=U^*\bar Q=\Sigma V^*$$ give rise to $N_0,N_1,N_2$ in the same way as $\bar P$ and $\bar Q$, but $$\tilde P=\left(\begin{array}{c}P \\H\end{array}\right), \quad \tilde Q=\left(\begin{array}{c}Q \\0\end{array}\right)$$ with $P,Q\in {\rm Mat}_{n-1,n}(\mathbb C).$ Since  $$\hat P:=\left(\begin{array}{c}P \\0\end{array}\right),\quad  \hat Q:=\left(\begin{array}{c}Q \\ H\end{array}\right)$$ also represent $N_0,N_1,N_2,$ and since  $$\hat Q^*\hat Q=  Q^* Q +H^*H= \tilde Q^*\tilde Q+ H^*H= \bar Q^*\bar Q+ H^*H,$$ maximality of $\bar Q^*\bar Q$ implies  $H=0.$ Thus   $P$ and $Q$ provide the desired representation of $N_0,N_1,N_2$.

For the converse, assume we have a representation for $N_0,N_1,N_2$ by $P,Q\in{\rm Mat}_{n-1,n}(\mathbb C)$. Then for each $z=a+bi\in\mathbb C$ with $a^2+b^2=1$, the matrix $$N_0+aN_1+bN_2=(zP+Q)^*(zP+Q)\geqslant 0$$ is singular. This implies $D=\mathcal S_1(N_0,N_1,N_2)$.
\end{proof}

Recall the definition of the matrices $M_1,M_2$ above. Additionally now let
\begin{equation*}
    M_3:=\begin{pmatrix}
        0 & i\\
        -i & 0
    \end{pmatrix}.
\end{equation*}
With these matrices we can explicitly describe $\mathcal{C}(D)^{\mathrm{lrg}}$ as a free spectrahedral shadow.

\begin{theorem}
    Let $N_0,N_1,N_2\in\mathrm{Her}_n(\mathbb{C})$. Then $(N_0,N_1,N_2)\in\mathcal{C}(D)^{\mathrm{lrg}}_n$ if and only if there exists a Hermitian matrix $N_3\in\mathrm{Her}_n(\mathbb{C})$ with
    \begin{equation*}
        N_0\otimes I_2+N_1\otimes M_1+N_2\otimes M_2+N_3\otimes M_3\geqslant 0.
    \end{equation*}
\end{theorem}
\begin{proof}
    Using the self-duality of $D$ and \Cref{thm:genspeccontain} we first observe that $(N_0,N_1,N_2)\in\mathcal{C}(D)^{\mathrm{lrg}}_n$ is equivalent to $D\subseteq\mathcal{S}_1(N_0,N_1,N_2)$. Given this, we can find matrices $P,Q\in\mathrm{Mat}_{r,n}(\mathbb{C})$ as in \Cref{thm:circcone}~($i$). Now we can choose $N_3:=P^*P-Q^*Q$ and set $$A:=\begin{pmatrix}(P+Q) & i(P-Q)\end{pmatrix}\in\mathrm{Mat}_{r,2n}(\mathbb{C}).$$ Then we have
    \begin{equation*}
        N_0\otimes I_2+N_1\otimes M_1+N_2\otimes M_2+N_3\otimes M_3=A^*A\geqslant 0.
    \end{equation*}
    For the converse, we can first write
    \begin{equation*}
        N_0\otimes I_2+N_1\otimes M_1+N_2\otimes M_2+N_3\otimes M_3=B^*B
    \end{equation*}
    with $B=\begin{pmatrix} U & V\end{pmatrix}\in\mathrm{Mat}_{r,2n}(\mathbb{C})$. When we now set $P:=(U-iV)/2$ and $Q:=(U+iV)/2$ we obtain a decomposition as in \Cref{thm:circcone}~($i$).
\end{proof}

We know  that $\mathcal C(D)^{\rm lrg}$ is not a free spectrahedron. However,  as explained above, we cannot exclude a priori that  there is one or at least  finitely many maximal free spectrahedra over $D$, that contain all the others.   Before we can decide this question, we introduce the following construction.
Let $P\subseteq D$ be  a regular polyhedral cone  with $r$ extreme rays, that all lie on the boundary of $D$. Let $\ell_1,\ldots,\ell_r\in\mathbb R[x_0,x_1,x_2]$ be the linear forms that define $P$ by inequalities, where we assume the coefficient of $x_0$ to be $1$ in all $\ell_j$, to make them unique. We  set $h=\ell_1\cdots \ell_r\in\mathbb R[x_0,x_1,x_2]$ and consider its partial  derivative $$\frac{\partial}{\partial x_0}h\in\mathbb R[x_0,x_1,x_2]$$ with respect to $x_0$. See \Cref{fig:ren} for an illustration  of the real zero set of $h$ and $\frac{\partial}{\partial x_0}h$ in the case of $r=7.$
The below proof was suggested to us by Mario Kummer. It significantly simplifies our previous proof. We are grateful to him for letting us include it here.
\begin{proposition}\label{prop:deriv}
For $r\geqslant 3$ odd,  $\frac{\partial}{\partial x_0}h$ factors into quadratic forms of the type $$\lambda x_0^2-x_1^2-x_2^2$$ with $\lambda\geqslant 1$, and $x_0^2-x_1^2-x_2^2$ is among these forms.
\end{proposition} 
\begin{proof}
 Since $h$ is invariant under rotations of angle $2\pi/r$ around the $x_0$-axis, it follows that
$q:=\frac{\partial}{\partial x_0}h$ has the same invariance. The space of homogeneous polynomials of degree $r-1$ that are invariant under this rotation has dimension $\frac{r+1}{2}$. This can be computed by elementary means, and also by Molien's theorem \cite{mol} (see for example also \cite{stu}).  Now whenever two $\ell_i$ have a common zero, the polynomial $q$ vanishes there as well. This gives additional $\frac{r-1}{2}$ many independent conditions on $q$, as one easily checks. So the space of all invariant polynomials of degree $r-1$ that vanish at these points is one-dimensional. Obviously, a suitable product of quadratic polynomials (as stated in the formulation of this proposition) belongs to this space. This proves the claim.
\end{proof}
\begin{figure}\includegraphics[scale=0.5]{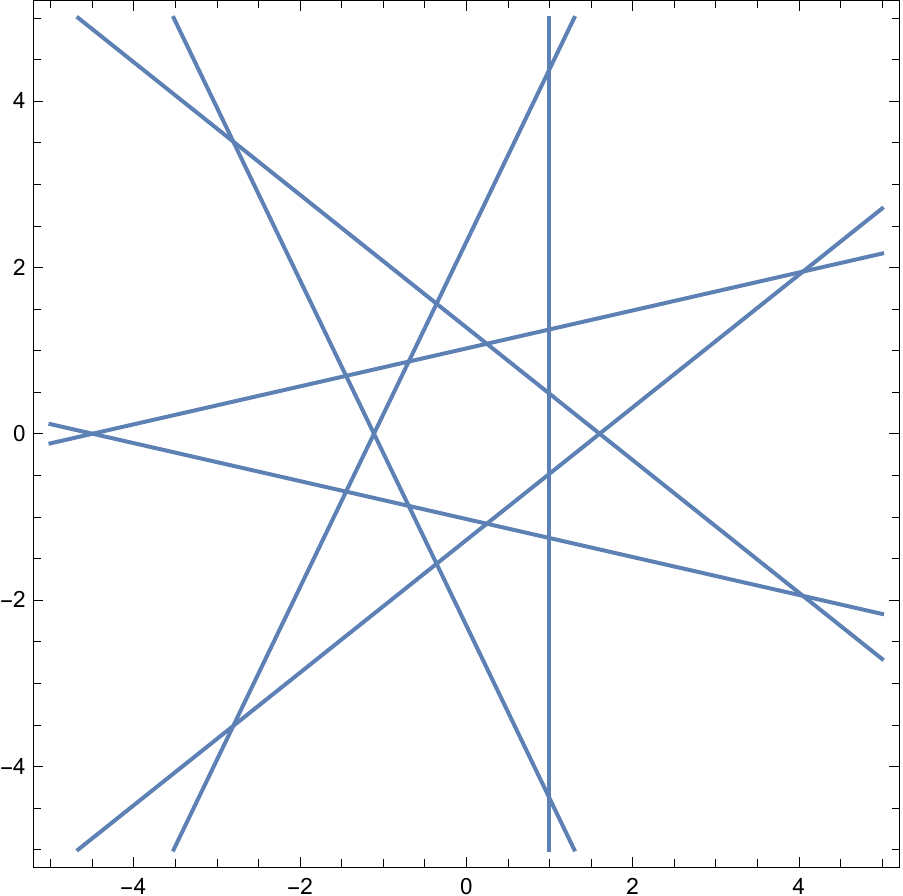}
\includegraphics[scale=0.5]{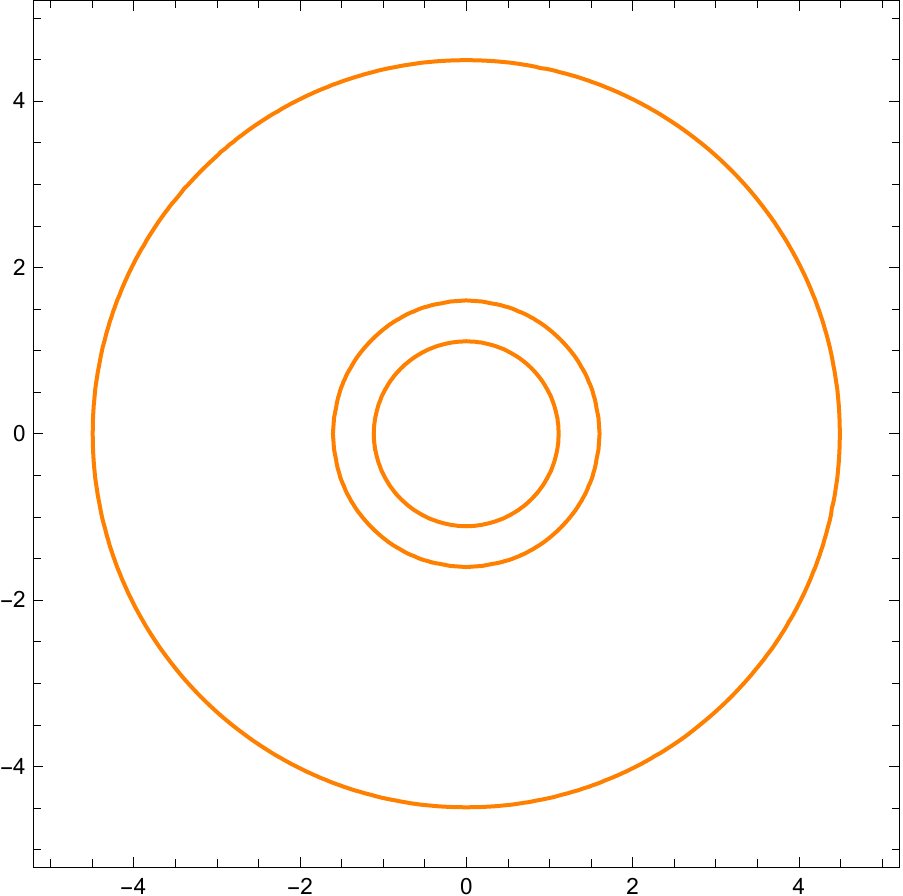}
\includegraphics[scale=0.5]{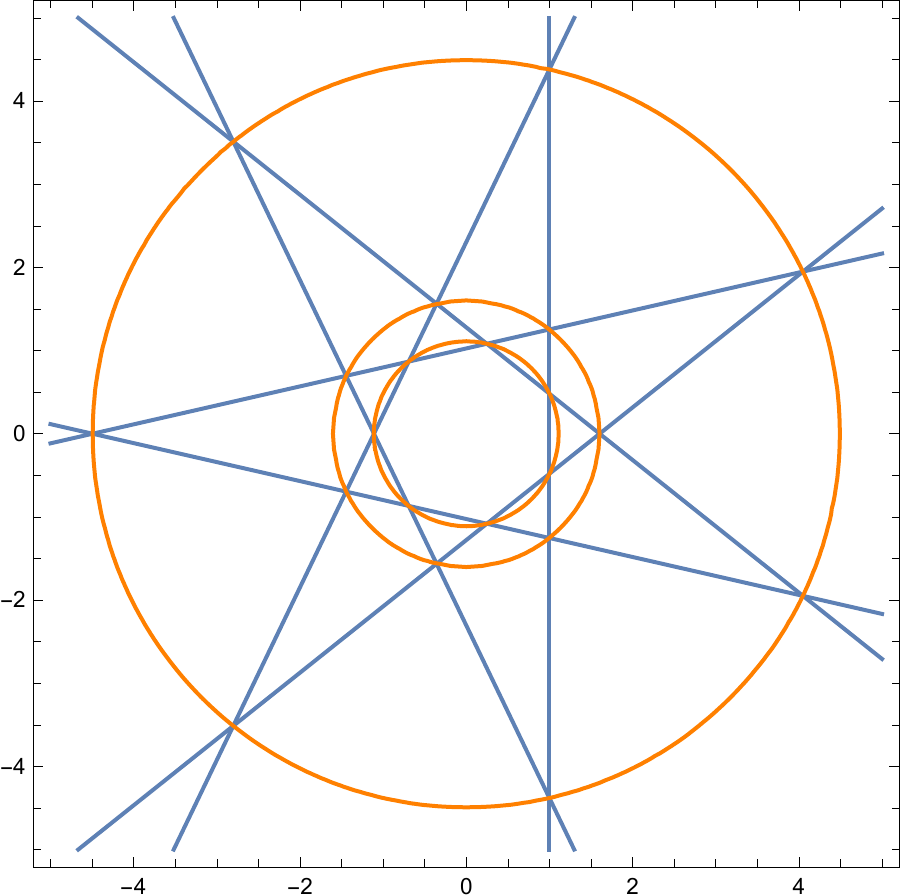}
\caption{From left to right: real zero set of $h, \frac{\partial}{\partial x_0}h$ and both combined, drawn on the affine section defined by $x_0=1$.}
\label{fig:ren}
\end{figure}

The following result is a strengthening of \Cref{thm:gencon} for the circular cone.
\begin{theorem}$(i)$ The union of all free spectrahedra over  $D$ covers the interior of $\mathcal C(D)^{\rm lrg}$.   $\mathcal C(D)^{\rm sml}$ is the intersection of all free numerical ranges over $D$.

$(ii)$ There is a point in $\mathcal C(D)^{\rm lrg}$ which is not contained in any free spectrahedron over $D$.

$(iii)$  It is \emph{not} the case that there exist finitely many free spectrahedra over $D$ which  contain all the other ones.
\end{theorem}
\begin{proof} ($i$) For each odd $r\geqslant 3$, the polyhedral cone $P$ from above is a spectrahedron, defined by the diagonal linear matrix inequality $$0\leqslant\left(\begin{array}{ccc}\ell_1(x_0,x_1,x_2) &  &  \\ & \ddots &  \\ &  & \ell_r(x_0,x_1,x_2)\end{array}\right)=x_0I_r +x_1L_1+x_2L_2,$$ and the same matrix inequality defines its maximal abstract operator system: $$\mathcal C(P)^{\rm lrg}=\mathcal S(I_r,L_1,L_2).$$ This is the easy direction of  \cite[Theorem 3.2]{fritz}, a direct observation. By \cite{san}, for $h=\ell_1\cdots\ell_r$ we know that $$\frac{\partial}{\partial x_0}h=r\cdot {\rm det}\left(x_0I_{r-1}+x_1N_1+x_2N_2\right)$$ with  $$N_i=\left(\begin{array}{c|c}I_{r-1} & 1\end{array}\right) L_i \left(\begin{array}{c}I_{r-1} \\1\end{array}\right)\in{\rm Her}_{r-1}(\mathbb C).$$  By \Cref{prop:deriv},  $\frac{\partial}{\partial x_0}h$ vanishes on $\partial D,$  but not in the interior of $D$. This implies  $$\mathcal S_1(I_{r-1}, N_1,N_2)=D$$ (see for example \cite{nepl} for more information on determinantal representations and spectrahedra). Since the $N_i$ arise as compressions of the $L_i,$ we have $$\mathcal C(P)^{\rm lrg}=\mathcal S(I_r,L_1,L_2)\subseteq \mathcal S(I_{r-1},N_1,N_2) \subseteq \mathcal C(D)^{\rm lrg}.$$ 
It is easy to see that each element from the interior of $\mathcal C(D)^{\rm lrg}$ already lies in $\mathcal C(P)^{\rm lrg}$ for some large enough $r$. This proves the first claim, and the statement about  $\mathcal C(D)^{\rm sml}$ follows from free duality. 

For  ($ii$)  consider the hermitian matrix tuple $$(A_0,A_1,A_2):=\left(\left(\begin{array}{cc}1 & 0 \\0 & 1\end{array}\right), \left(\begin{array}{cc}0 & 0 \\0 & 1\end{array}\right),\left(\begin{array}{cc}0 & 1/\sqrt{2} \\ 1/\sqrt{2} & 0\end{array}\right)\right)\in \mathcal C(D)_2^{\rm lrg}.$$ Now assume that $D=\mathcal S_1(I_n, N_1,N_2)$ for certain $N_1,N_2\in{\rm Her}_n(\mathbb C).$ Note that taking $I_n$ as our first matrix  is not a restriction, since  every defining linear matrix inequality for $D$ can be reduced to such a normalized one, see for example \cite{rago, nepl}. We can further  assume that $N_2$ is diagonal: $$N_2=\left(\begin{array}{ccc}d_1 & 0 & 0 \\0 & \ddots & 0 \\0 & 0 & d_n\end{array}\right).$$ Since $(1,1,\lambda)\in D$ if and only if $\lambda=0$, we obtain 
\begin{align}\label{eins}N_1+\left(\begin{array}{ccc}1+\lambda d_1 & 0 & 0 \\0 & \ddots & 0 \\0 & 0 & 1+\lambda  d_n\end{array}\right) \geqslant 0 \ \Leftrightarrow\ \lambda=0.\end{align}
Now assume that $$0\leqslant A_0\otimes I_n + A_1\otimes N_1+ A_2\otimes N_2=\left(\begin{array}{cc}I_n & \frac{1}{\sqrt{2}} N_2 \\ \frac{1}{\sqrt{2}}N_2 & I_n+N_1\end{array}\right)$$ holds, which is equivalent to \begin{align}\label{zwei}0\leqslant I_n +N_1-\frac{1}{2}N_2^2=N_1+\left(\begin{array}{ccc}1-\frac{1}{2}d_1^2 & 0 & 0 \\0 & \ddots & 0 \\0 & 0 & 1-\frac{1}{2}d_n^2\end{array}\right).\end{align} Consider the following equation: 
\begin{align*}&N_1+\left(\begin{array}{ccc}1+\lambda d_1 & 0 & 0 \\0 & \ddots & 0 \\0 & 0 & 1+\lambda d_n\end{array}\right) \\ & = \underbrace{N_1+ \left(\begin{array}{ccc}1-\frac{1}{2}d_1^2 & 0 & 0 \\0 & \ddots & 0 \\0 & 0 & 1-\frac{1}{2}d_n^2\end{array}\right)}_{\geqslant 0 \mbox{ by } (\ref{zwei})} +   \left(\begin{array}{ccc}\frac12 d_1^2+\lambda d_1 & 0 & 0 \\0 & \ddots & 0 \\0 & 0 & \frac12 d_n^2+\lambda d_n\end{array}\right).\end{align*}  If $\vert\lambda\vert$ is small, the matrix on the very right is positive semidefinite, which contradicts (\ref{eins}).

Finally, ($iii$) is a direct consequence of ($i$) and ($ii$).
\end{proof}

As explained above, we now obtain the following conclusion about linear matrix inequalities defining the circular cone.
\begin{corollary} There exists a sequence of linear matrix inequality definitions of the circular cone $D$,  of which none is a sum of compressions of the previous one.
\end{corollary}
\subsection{The Cone of Positive Semidefinite Matrices}\label{ssec:psd}
Throughout this section we consider the space $\mathcal V={\rm Mat}_d(\mathbb C)$ with the usual involution, so that \begin{align*}{\rm Mat}_s(\mathcal V)&={\rm Mat}_s(\mathbb C)\otimes{\rm Mat}_d(\mathbb C)={\rm Mat}_{sd}(\mathbb C) \\{\rm Her}_s(\mathcal V)&={\rm Her}_s(\mathbb C)\otimes{\rm Her}_d(\mathbb C)={\rm Her}_{sd}(\mathbb C).\end{align*} We will always use the cone $$P_d:=\{A\in {\rm Her}_d(\mathbb C)\mid A\geqslant 0\}$$ of positive semidefinite matrices as cone at level one. Then $\mathcal C(P_d)^{\rm sml}$ consists of so-called {\it separable} matrices, and  $\mathcal C(P_d)^{\rm lrg}$ of {\it block-positive} matrices. Neither of the two is a free spectrahedron or a free numerical range, see for example \cite{bene} for more detailed explanations. There are two particularly interesting abstract operator systems $${\rm Psd}_d=\left( {\rm Psd}_{d,s}\right)_{s\geqslant 1}\ \mbox{ and }\ {\rm Psd}^\Gamma_d=\left( {\rm Psd}_{d,s}^\Gamma\right)_{s\geqslant 1}$$ over $P_d$, defined by \begin{align*}{\rm Psd}_{d,s}&:=\left\{ A\in {\rm Her}_s(\mathbb C)\otimes{\rm Her}_d(\mathbb C)\mid A\geqslant 0\right\}\\
{\rm Psd}^\Gamma_{d,s}&:=\left\{ \sum_i A_i\otimes B_i\in {\rm Her}_s(\mathbb C)\otimes{\rm Her}_d(\mathbb C)\mid \sum_i A_i\otimes B_i^T\geqslant 0\right\}.\end{align*}
${\rm Psd}_d$ is called the system of positive matrices, and ${\rm Psd}^\Gamma_d$ the system of matrices with positive partial transpose.
Both  are easily seen to be free spectrahedra,  defined by linear matrix inequalities of size $d$. In \cite{bene} it was shown that they are free dual to each other (note the slight difference in the definition of the free dual there!). In particular, both are also free numerical ranges, generated by elements belonging to their $d$-th level. Among  free spectrahedra defined by linear  matrix inequalities of relatively small size, they are in fact the only maximal ones.

\begin{theorem}\label{thm:psd}
Let $\mathcal C$ be an abstract operator system over $P_d$. 

$(i)$ If $\mathcal C$ is a free spectrahedron, defined by a linear matrix inequality of size at most $2d-2$, then $$\mathcal C\subseteq {\rm Psd}_d\ \mbox{ or }\ \mathcal C\subseteq {\rm Psd}_d^\Gamma.$$

$(ii)$ If $\mathcal C$ is a free numerical range, generated by an element from $\mathcal C_{2d-2}$,  then  $${\rm Psd}_d\subseteq \mathcal C\ \mbox{ or } \ {\rm Psd}_d^\Gamma\subseteq \mathcal C.$$
\end{theorem}
\begin{proof}
($i$) By assumption, there is a concrete realization $$\varphi\colon{\rm Mat}_d(\mathbb C)\to{\rm Mat}_r(\mathbb C)$$ of $\mathcal C$, with $r\leqslant 2d-2$. As explained above, we can assume $\varphi(I_d)=I_r.$ From $\mathcal C_1=P_d$ we see that $$ A\geqslant 0 \ \Leftrightarrow\ \varphi(A)\geqslant 0 $$ holds for all $A\in{\rm Her}_d(\mathbb C)$. By \cite[Theorem 2.3]{chlipo} there exists a unitary $U\in{\rm Mat}_r(\mathbb C)$ and  a unital positive $*$-linear map $\psi\colon {\rm Mat}_d(\mathbb C)\to {\rm Mat}_{r-d}(\mathbb C),$ such that either $$U^*\varphi(A)U=\left(\begin{array}{cc}A & 0 \\0 & \psi(A)\end{array}\right)$$ or $$U^*\varphi(A)U=\left(\begin{array}{cc}A^{T} & 0 \\0 & \psi(A)\end{array}\right)$$ holds for all $A\in{\rm Mat}_d(\mathbb C)$. In the first case we have $\mathcal C\subseteq {\rm Psd}_d,$ in the second $\mathcal C\subseteq {\rm Psd}_d^\Gamma.$ 
($ii$) follows from ($i$) by free duality.
\end{proof}

A linear matrix inequality definition of size $r$ for $P_d$ is the same as a $*$-linear map $$\varphi\colon {\rm Mat}_d(\mathbb C)\to{\rm Mat}_r(\mathbb C)$$ with $$A\geqslant 0 \ \Leftrightarrow\ \varphi(A)\geqslant 0$$ for all $A\in{\rm Mat}_d(\mathbb C)$. After possibly splitting off zero blocks, we can further assume $\varphi(I_d)=I_r.$ In \cite{chlipo} it was proven that this is equivalent to $\varphi$ being a {\it unital  $*$-isometry} with respect to the operator/spectral norm on the matrix algebras. So classifying linear matrix inequalities defining  $P_d$ is the same as classifying unital $*$-isometries between matrix algebras. 

Some  isometries are easy, meaning  there exists a unitary $U\in{\rm Mat}_r(\mathbb C)$ such that $$U^*\varphi(A)U= \left(\begin{array}{cc}A & 0 \\0 & \psi(A)\end{array}\right)\quad \mbox{  or }\quad U^*\varphi(A)U=\left(\begin{array}{cc}A^T & 0 \\0 & \psi(A)\end{array}\right)$$ holds for all $A\in{\rm Mat}_d(\mathbb C)$. All easy isometries define free spectrahedra contained in either  ${\rm Psd}_d$ or ${\rm Psd}_d^\Gamma.$ Theorem 2.3 from \cite{chlipo}, that we have used above, says that for $r\leqslant 2d-2$, all isometries are easy, but this fails for $r\geqslant 2d-1$. The following result provides alternative characterization of easy isometries.

\begin{theorem}\label{thm:easy} For a unital $*$-isometry $\varphi\colon {\rm Mat}_d(\mathbb C)\to{\rm Mat}_r(\mathbb C),$ the following are equivalent: 
\begin{itemize}
\item [$(i)$] The free spectrahedron defined by $\varphi$ is contained in ${\rm Psd}_d$ or ${\rm Psd}_d^\Gamma$.
\item[$(ii)$] There exists $V\in{\rm Mat}_{r,d}(\mathbb C),$ such that either $$V^*\varphi(A)V=A\ \mbox{ or } \ V^*\varphi(A)V=A^T$$ holds for all $A\in{\rm Mat}_d(\mathbb C)$.
\item[$(iii)$] $\varphi$ is easy, i.e.\ there exists a unitary $U\in{\rm Mat}_{r}(\mathbb C),$ such that either $$U^*\varphi(A)U=\left(\begin{array}{cc}A & 0 \\0 & \psi(A)\end{array}\right)\ \mbox{ or } \ U^*\varphi(A)U=\left(\begin{array}{cc}A^{T} & 0 \\0 & \psi(A)\end{array}\right)$$ holds for all $A\in{\rm Mat}_d(\mathbb C)$
\item[$(iv)$] There are unit vectors $h_1,\ldots, h_d\in\mathbb C^r$ such that either $$\varphi(E_{ij})h_s=\delta_{js}\cdot h_i\ \mbox{ or }\ \varphi(E_{ij})h_s=\delta_{is}\cdot h_j$$ holds for all $i,j,s.$
\end{itemize} 
\end{theorem}
\begin{proof} It is clear that ($iii$) implies ($ii$)  and ($ii$) implies ($i$). Now assume ($i$) holds, and further assume the free spectrahedron is contained in ${\rm Psd}_d$, without loss of generality. By \Cref{thm:cont} there exist finitely many $V_1,\ldots, V_\ell\in{\rm Mat}_{r,d}(\mathbb C)$ such that $$\sum_j V_j^*\varphi(A)V_j=A$$ holds for all $A\in{\rm Mat}_d(\mathbb C)$. Since $\varphi$ is positive, it is immediate that $$A\geqslant 0 \ \Leftrightarrow\ V_j^*\varphi(A)V_j\geqslant 0 \mbox{ for all } j=1,\ldots, \ell$$ holds for all $A$, i.e.\ the spectrahedron $P_d$ is the intersection of the finitely many spectrahedra defined by  $V_j^*\varphi(A)V_j\geqslant 0.$ Since the boundary of $P_d$ is defined by the vanishing of the real irreducible polynomial $\det(A)$, already one of the conditions $V_j^*\varphi(A)V_j\geqslant 0$ alone defines $P_d.$ So to the corresponding mapping $$A\mapsto V_j^*\varphi(A)V_j$$ we can apply \cite[Theorem 2.3]{chlipo} and obtain ($ii$). Now assume without loss of generality that the first case of ($ii$) holds. Extend $V$ to a unitary $U$, by adding columns. Then $$U^*\varphi(A)U=\left(\begin{array}{cc}A & \gamma(A)^* \\ \gamma(A) & \psi(A)\end{array}\right)$$ holds for all $A\in{\rm Mat}_d(\mathbb C)$. For any $v\in\mathbb C^d$ we know that $$0\leqslant U^*\varphi(vv^*)U=\left(\begin{array}{cc}vv^* & \gamma(vv^*)^* \\\gamma(vv^*) & \psi(vv^*)\end{array}\right)$$ holds, which implies that the kernel of $\gamma(vv^*)$ contains the orthogonal complement of $v$. On the other hand, if $v^*v=1$, then also $$0\leqslant U^*\varphi(I_d-vv^*)U=I_r-U^*\varphi(vv^*)U=\left(\begin{array}{cc}I_d-vv^* & -\gamma(vv^*)^* \\-\gamma(vv^*) & I_{r-d}-\psi(vv^*)\end{array}\right)$$ holds, which implies $\gamma(vv^*)v=0$. So $\gamma(vv^*)=0$, which implies $\gamma(A)=0$ for all $A\in{\rm Mat}_d(\mathbb C)$. This proves ($iii$).  If  ($iii$) holds, we set $h_i:=Ue_i$ for $i=1,\ldots, d$ and obtain ($iv$). Conversely if ($iv$) holds, we let the $h_i$ be the first $d$ columns of a matrix, and extend to a unitary $U\in{\rm Mat}_r(\mathbb C)$. This proves ($iii$).
\end{proof}

\begin{example}
The following example is from \cite{chlipo}. For $d\geqslant 2$ consider 
\begin{align*}\varphi\colon {\rm Mat}_d\mathbb C)&\to {\rm Mat}_{2d-1}(\mathbb C)\\
\left(\begin{array}{cc}a & b^t \\c & B\end{array}\right)&\mapsto\left(\begin{array}{ccc}a & \frac{1}{\sqrt{2}}b^t & \frac{1}{\sqrt{2}}c^t \\\frac{1}{\sqrt{2}}c & B & 0 \\\frac{1}{\sqrt{2}}b & 0 & B^t\end{array}\right).
\end{align*}
One checks that $$A\geqslant 0\ \Leftrightarrow\ \varphi(A)\geqslant 0$$ holds, so $\varphi$ gives rise to an operator system $\mathcal C$ on ${\rm Mat}_d(\mathbb C)$ with $P_d$ at level one. In \cite{chlipo} it was shown that $\varphi$ is not one of the easy isometries. In view of \Cref{thm:easy}, the free spectraehdron defined by $\varphi$ is not contained in ${\rm Psd}_d$ or ${\rm Psd}_d^\Gamma$.

Explicitly, we look at the level $2$ cone of $\mathcal C,$ and compare it to ${\rm Psd}_{d,2}$ and ${\rm Psd}_{d,2}^\Gamma$. We restrict to the affine subspace of ${\rm Her}_2(\mathbb C)\otimes {\rm Her}_d(\mathbb C)={\rm Her}_d({\rm Her}_2(\mathbb C))$  of matrices of the form  $$A(x,y):=
\left(\begin{array}{ccccc}\left(\begin{array}{cc}3 & 0 \\0 & 2\end{array}\right) & \left(\begin{array}{cc}0 & x \\y & 0\end{array}\right) & 0 & \cdots & 0 \\\left(\begin{array}{cc}0 & \overline y \\\overline x & 0\end{array}\right) & \left(\begin{array}{cc}1 & 0 \\0 & 1\end{array}\right) & 0 &&\\0 & 0 & 0 &   &  \vdots\\ \vdots &   &   & \ddots &   \\ 0 & &  \cdots  &   & 0\end{array}\right).$$
 Then $$A(x,y)\in\mathcal C_2 \Leftrightarrow \left(\begin{array}{cccccc}3 & 0 & 0 & \tfrac{x}{\sqrt{2}} & 0 & \tfrac{\overline y}{\sqrt{2}} \\0 & 2 & \tfrac{y}{\sqrt{2}} & 0 & \tfrac{\overline x}{\sqrt{2}} & 0 \\0 & \tfrac{\overline y}{\sqrt{2}} & 1 & 0 & 0 & 0 \\\tfrac{\overline x}{\sqrt{2}} & 0 & 0 & 1 & 0 & 0 \\0 & \tfrac{x}{\sqrt{2}} & 0 & 0 & 1 & 0 \\\tfrac{y}{\sqrt{2}} & 0 & 0 & 0 & 0 & 1\end{array}\right)\geqslant 0.$$ A direct computation for $(x,y)\in\mathbb R^2$ thus shows $$A(x,y)\in \mathcal C_2 \ \Leftrightarrow x^2+y^2\leqslant 4.$$ Similarly we compute $$A(x,y)\in {\rm Psd}_{d,2} \ \Leftrightarrow \ x^2\leqslant 3 \wedge  y^2\leqslant 2$$ and $$A(x,y)\in {\rm Psd}_{d,2}^\Gamma \ \Leftrightarrow \ x^2\leqslant 2 \wedge y^2\leqslant 3.$$ The following picture shows these affine sections of $\mathcal C_2$ (purple), ${\rm Psd}_{d,2}$ (orange), and  ${\rm Psd}_{d,2}^\Gamma$ (yellow):
 \begin{center}
 \includegraphics[scale=0.5]{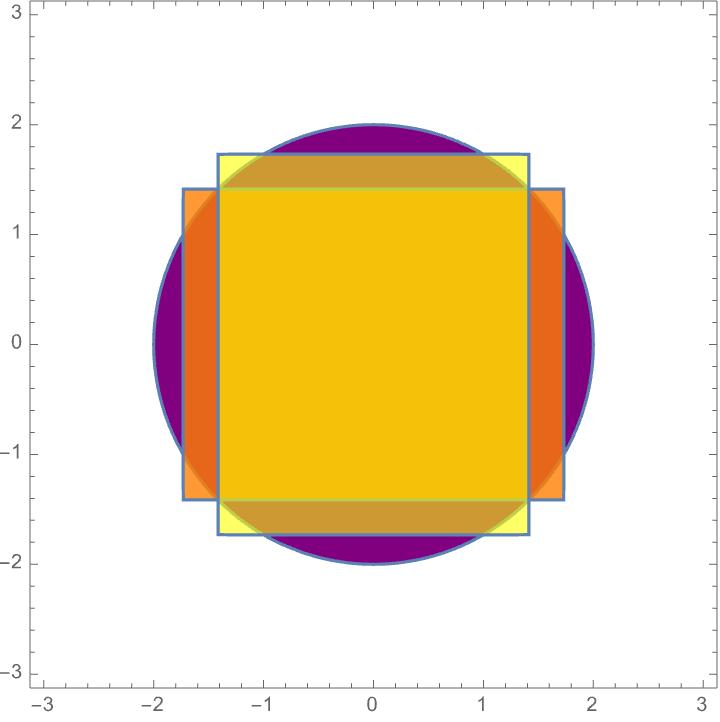}
 \end{center}
 We see that $\mathcal C$ is not  even contained in the union  of ${\rm Psd}_d$ and ${\rm Psd}_d^\Gamma$, and also does not contain ${\rm Psd}_d$ or ${\rm Psd}_d^\Gamma$.  \end{example}

\bibliographystyle{plain}
\bibliography{ref.bib}

\end{document}